\documentclass[12pt]{amsart}
\usepackage{graphicx} 

\usepackage[utf8]{inputenc}
 \usepackage[english]{babel}

\usepackage[margin=30mm,paperwidth=216mm, paperheight=303mm]{geometry}

\usepackage{etoolbox}
\patchcmd{\section}{\scshape}{\bfseries}{}{}
\makeatletter
\renewcommand{\@secnumfont}{\bfseries}
 
 \usepackage{graphicx}
 \usepackage{amsthm}
 \usepackage{amssymb}
 \usepackage{amsfonts}
 \usepackage{mathrsfs}  
 \usepackage{amsmath}
 \usepackage{ dsfont }
 \usepackage{mathtools}
 \usepackage{bm}
 \usepackage{yfonts}
 \usepackage{xcolor}
 \usepackage[all]{xy}
 \usepackage{enumerate}
 \usepackage{stmaryrd}
 \usepackage{pdfpages}
 \usepackage{csquotes}
 \usepackage[pdfborderstyle={/S/U/W 1}]{hyperref} 

\usepackage[backend=bibtex,style=alphabetic, maxbibnames=99]{biblatex}
\usepackage{listings}
\theoremstyle{plain}
\newtheorem{thm}{Theorem}[section]
\newtheorem*{thm*}{Theorem}

\newtheorem*{satz*}{Satz}
\newtheorem{prop}[thm]{Proposition}
\newtheorem*{prop*}{Proposition}
\newtheorem{lem}[thm]{Lemma}
\newtheorem*{lem*}{Lemma}

\newtheorem*{koro*}{Corollary}

\theoremstyle{definition}
\newtheorem{defi}[thm]{Definition}
\newtheorem*{defi*}{Definition}

\newtheorem{taggedtheoremx}{Condition}
\newenvironment{taggedcondition}[1]
 {\renewcommand\thetaggedtheoremx{#1}\taggedtheoremx}
 {\endtaggedtheoremx}

\theoremstyle{remark}
\newtheorem{rem}[thm]{Remark}
\newtheorem*{rem*}{Remark}
\newtheorem{example}[thm]{Example}
\newtheorem*{example*}{Example}

\numberwithin{equation}{section}

\newcommand{\fr}{\mathfrak}

\newcommand{\RR}{\mathbb{R}}

\newcommand{\ZZ}{\mathbb{Z}}
\newcommand{\NN}{\mathbb{N}}

\newcommand\sumst{\mathop{{\sum}^\ast}\limits}

\newcommand{\mods}[1]{\,(\mathrm{mod}\,{#1})}

\renewcommand{\geq}{\geqslant}
\renewcommand{\leq}{\leqslant}

\DeclareMathOperator{\vol}{vol}

\DeclareMathOperator{\rank}{rank}

\title{Solving quadratic forms in restricted variables with the circle method}
\author{Mieke Wessel and Svenja zur Verth}

\begin{document}
\begin{abstract}
 Let $f(\bm x)$ be a non-singular quadratic form with sufficiently many mixed terms and $t$ an integer. For a sequence of weights $\mathcal A$ we study the number of weighted solutions to $f(\bm x) = t$. In particular, we give conditions on both $\mathcal A$ and $f$ such that we can use the circle method to count such solutions of bounded height.
\end{abstract}
\maketitle

\section{Introduction}
Let $f(\bm x) \in \ZZ[x_1, \ldots, x_s]$ be a non-singular quadratic form and $\mathcal A$ a sequence $(a_x)_{x \in \NN}$ with $a_x \in \RR_{\geq 0}$ such that the sum over $\mathcal A$ diverges. The goal of this paper is to give general conditions on $\mathcal A$ and $f$ for which we can carry out the circle method to count the weighted number of solutions of bounded height to $f(\bm x) = t$, $t\in\ZZ$, where the weights equal $\prod_{i=1}^s a_{x_i}$. In other words, we will give conditions under which we can determine an asymptotic formula for
\begin{equation*}
    R_{f, t}(X, \mathcal A) := \sum_{0 \leq \bm x \leq X} \mathds 1_{[f(\bm x) = t]} \prod_{i=1}^s a_{x_i}.
\end{equation*}
For some specific choices of $\mathcal A$ this counting function has already been studied thoroughly. For instance, if we choose $a_x = 1$ for all $x \in \NN$ the value of $R_{f, t}(X, \mathcal A)$ is the number of positive integral solutions to $f(\bm x) = t$. This is a classical problem and it is known due to Heath-Brown \cite{HeathBrown1996} that an asymptotic formula exists whenever the number of variables $s$ is greater or equal to $3$. By work of Baker \cite{Baker1, Baker2}, the asymptotic formula for the case $a_x = \mu(x)^2$, which is the indicator function of the squarefree numbers, is known, also whenever $s \geq 3$. Another important example is motivated by the question if $f(\bm x) = t$ has any solutions in the primes. This has been researched by Liu and Zhao in \cite{Liu2011integralpoints} and \cite{ZHAO_2016}. 
Rather than taking the indicator function for primes they considered the weights $a_x = \Lambda(x)$, where $\Lambda(x)$ is the von Mangoldt function. They conclude that whenever $f$ has sufficiently many mixed terms, the asymptotic formula can be determined. The meaning of sufficiently many is similar to our Condition \ref{CL1} below, and consequently implies that $s \geq 9$ or $s\geq 10$. 

One thing that stands out about these examples is that in each case the sequence $\mathcal A$ is far from random. It has a lot of arithmetic structure and the series $A(X) := \sum_{x=1}^{X} a_x$ has a nice asymptotic expression as $X\in \RR$ grows. That raises the question, which minimal conditions on $\mathcal A$ and $f$ are sufficient to be able to apply the circle method to get an asymptotic formula for $R_{f, t}(X, \mathcal A)$. To answer this question we build upon the work of Biggs and Brandes \cite{biggs2023minimalistversioncirclemethod}, who studied a very similar question for several variations of Waring's problem. In our case, however, we will consider forms that have many mixed terms.

To make the idea of `many mixed terms' more precise we introduce $F = (F_{i, j})_{1\leq i, j \leq s} \in M_{s\times s}(\ZZ)$, the symmetric matrix of $f$ such that 
$$f(\bm x) = \sum_{i=1}^s \sum_{j=1}^s F_{i, j}x_ix_j = \bm x^T F \bm x.$$ 
Define $n = \lfloor s/2 \rfloor$ and let $F_1$ be the $n \times n$ matrix, $F_2$ the $(s - n) \times n$ matrix, $F_3$ the $n \times (s-n)$ matrix and $F_4$ the $(s-n)\times (s-n)$ matrix such that
$$F = \begin{pmatrix}
    F_1 & F_2 \\
    F_3 & F_4
\end{pmatrix}.$$ 
\begin{taggedcondition}{$\textup{L}_1$}\label{CL1}
    We say that $f$ meets Condition $\text{L}_1$ if we can permute the variables $x_1, \ldots, x_s$ such that
    \begin{equation*}\label{EQL1}
        \rank(F_2) \geq 5.
    \end{equation*}
\end{taggedcondition}
\begin{rem}
    In particular Condition \ref{CL1} implies that $s \geq 10$ and that $f$ has at least $5$ mixed terms.
\end{rem}

Apart from needing a positive non-singular solution to $f(\bm x)=0$ to guarantee positivity of the singular integral, Condition \ref{CL1} will be the only condition on $f$. All the other conditions are on the sequence $\mathcal A$. From now on we will write $L$ to denote $\log(X)$ and we recall that $A(X) = \sum_{x=1}^{X} a_x, X \in \RR$. 
\begin{taggedcondition}{$\textup{L}_2$}\label{CL2}
    We say that $\mathcal A$ meets Condition $\text{L}_2$ if there exists a constant $c\geq 0$ such that for all $X \geq 1$ it holds that
\begin{equation*}
\sum_{x = 1}^X a_x^2 \ll A(X)^2X^{-1}L^c.
\end{equation*}
\end{taggedcondition}
\begin{rem}
Note that by Cauchy-Schwarz we have the lower bound $$A(X)^2X^{-1}\leq \sum_{x = 1}^X a_x^2.$$
\end{rem}
\begin{rem}\label{R L2withindicatorfct}
    When $a_x$ is an indicator function we have that $\sum_{x \leq X}{a_x^2} = A(X)$ and therefore, 
    $$A(X) \gg XL^{-c}$$
    whenever Condition \ref{CL2} holds.
\end{rem}

 To apply the circle method it is necessary that the weights $a_x$ are distributed in a consistent way over residue classes modulo $q$ as $X$ grows. For $q \in \NN$ and the residue class $h$ modulo $q$ we define $A(q, h; X) := \sum_{\substack{1\leq x \leq X \\ x \equiv h \mods q}} a_x$.

\begin{taggedcondition}{D}\label{CD}
Let $Q_D = Q_D(X)$ be an increasing function of $X$. We refer to a set $\mathcal A$ as distributed to the level $Q_D$ when for all $X>1$ and for each $q \leq Q_D$ and $h \mods q$ there is a real number $\kappa(q, h) \geq 0$ such that the asymptotic formula
\begin{equation*}
A(q, h; X) = \kappa(q, h)A(X) + E(X)
\end{equation*}
is satisfied with an error term $E(X) = O(A(X)L^{-m})$ for all $m \in \NN$. The implicit constant is allowed to depend on $m$.  
\end{taggedcondition}

\begin{rem}
    A similar condition was first stated in \cite{Bruedern13inBB}. Our condition is a bit stronger than the equivalent one in \cite{biggs2023minimalistversioncirclemethod} -- also called Condition (D) -- where the error term $E(X)$ is allowed to be of order $o(A(X))$.  
\end{rem}

\begin{rem}
    Note that Condition \ref{CD} implies 
    $$\sum_{h\mods q}\kappa(q,h)=1.$$
\end{rem}

The following condition states that $\kappa$ has some multiplicative structure, which is needed to interpret the constant of the asymptotic formula as a product of local densities.

\begin{taggedcondition}{C}\label{CC}
    Suppose that $\mathcal{A}$ is distributed to the level $Q_D=Q_D(X)$. Let $q, q', h, h' \in \NN$. We say $\mathcal A$ satisfies Condition C if the coefficients $\kappa(q,h)$ satisfy 
    \begin{equation*}
 \kappa(qq', qh'+q'h)=\kappa(q,q'h)\kappa(q',qh') \text{ for all }(q,q')=1 \text{ such  that } qq' < Q_D.
\end{equation*}
\end{taggedcondition}

Finally, our last condition states that $\kappa$ is close to being uniformly distributed. We need this assumption to ensure that product of local densities converges.

\begin{taggedcondition}{K}\label{CK}
For $q\rightarrow\infty$ there exists a constant $c\geq 0$ independent of $q$ such that 
    \begin{equation*}  
 \kappa(q,h)\ll \frac{(\log q)^c}{q} .  \end{equation*}
\end{taggedcondition}

Our main result states that under the above conditions we can give an asymptotic formula for $R_{f, t}(X, \mathcal A)$.

\begin{thm}\label{Tmain}
    Assume Conditions \emph{\ref{CC}, \ref{CD}, \ref{CK}, \ref{CL1}} and \emph{\ref{CL2}} hold and that $A(X)$ is smoothly approximable (see Definition \ref{Dsmthaprox} below). Furthermore, let $K>0$ be any real number and define $\mathfrak I(X)$ as in Equation (\ref{EQ truncated sing series}). Then
    $$R_{f, t}(X, \mathcal A) = \mathfrak S \mathfrak{I}(X) + O(A(X)^sX^{-2}L^{-K} + \mathfrak I (X) L^{-K}).$$
    Here $\mathfrak S$ is a non-negative constant that can be interpreted as the product of local $p$-adic densities for all primes $p$. 
\end{thm} 
\begin{rem}
    A typical example one can think of for $A(X)$ being smoothly approximable is $A(X) = XL^k + O(XL^{-m})$ for $k \in \RR$ fixed and any $m$.
\end{rem}
To also interpret $\mathfrak I(X)$ we consider a slightly different related counting problem. Let $\mathfrak B$ be a compact subset of $\RR_{>0}^s$ and define $R_{f, t}(X, \mathcal A, \mathfrak B)$ as the weighted number of solutions inside $X\mathfrak B := \{ X\bm v \mid \bm v \in \mathfrak B\}$, more precisely
$$R_{f, t}(X, \mathcal A, \mathfrak B) := \sum_{\bm x \in X\mathfrak B} \mathds 1_{\{0\}}(f(\bm x) - t)\prod_{i=1}^s a_{x_i}.$$
Note that $R_{f, t} (X, \mathcal A, [0, 1]^s) = R_{f, t}(X, \mathcal A).$
\begin{thm}\label{Tmainbox}
Assume we are in the same setting as in Theorem \ref{Tmain} and that $f(\bm x) = 0$ has a solution in $\RR_{>0}^s$. Then there are compact sets $\mathfrak C_1 \subset \mathfrak C_2$ such that for each $X$ there exists $\mathfrak B$, a suitable compact set with $ \mathfrak C_1 \subset \mathfrak B \subset \mathfrak C_2$, for which the following holds.

Let $\mathfrak I_{\mathfrak B}(X)$ be as in Equation (\ref{EQsingintbox}). Then
$$R_{f, t}(X, \mathcal A, \mathfrak B) = \mathfrak S \mathfrak I_{\mathfrak B}(X) + O\left(A(X)^sX^{-2}L^{-K}\right),$$
where $\mathfrak S$ has the same value as in Theorem \ref{Tmain}. Furthermore
    $$A(X)^sX^{-2} \ll \mathfrak I_{\mathfrak B}(X) \ll A(X)^sX^{-2},$$
    and the value $\frac{X^2}{A(X)^s}\mathfrak I_{\mathfrak B}(X)$ can be interpreted as a local density over the reals inside the compact set $\mathfrak B$.
\end{thm}
We will prove these results using the circle method. For $\bm x \in \ZZ^s$, we write $a_{\bm x}$ for $\prod_{i=1}^s a_{x_i}$. Hence,
\begin{equation}\label{EQ countingtocircle}
    R_{f, t}(X, \mathcal A) = \int_0^1\sum_{0 \leq \bm x \leq X} a_{\bm x}e(\alpha (f(\bm x)- t)) \text{d}\alpha.
\end{equation}
Indeed, one would then expect an asymptotic expression of the form $CX^{-2}A(X)^s$, where $C \geq 0$ can be interpreted as a product over the densities of solutions in the reals and $p$-adics.

In Section \ref{S setup and notation} we set up the notation and discuss smooth approximability of $A(X)$. 
In Section \ref{S weyl estimate} we give a Weyl-like estimate for the exponential sum with which we can then bound the minor arcs. This is done under Conditions \ref{CL1} and \ref{CL2}, where the first condition ensures that we can apply a Cauchy-Schwarz inequality before we start Weyl-differencing and the second condition implies that applying Cauchy-Schwarz gives strong upper bounds. 

In Section \ref{S Major arcs} we treat the major arcs in similar fashion as Biggs and Brandes did in \cite{biggs2023minimalistversioncirclemethod}. Here we will need Condition \ref{CD} to make sure we have consistent behavior as $X$ goes to infinity.

The analysis of the minor and major arcs already gives us an expression for $R_{f, t}(X, \mathcal A)$ with the correct error term. To interpret the main term we consider the $p$-adic and real behavior of $f$ separately in Sections \ref{S singular series} and \ref{S singular integral}. First we interpret our so-called singular series as the product of $p$-adic densities using Conditions \ref{CK} and \ref{CC}. Then we use our bounds on $A(X)$ to conclude the claims made about $\mathfrak I_{\mathfrak B}(X)$ in Theorem \ref{Tmainbox}.

Finally, in Section \ref{S main thm}, we prove Theorems \ref{Tmain} and \ref{Tmainbox} and in Section \ref{SExample} we give some examples of sequences $\mathcal A$ to which we can apply our theorem.
The purpose of this section is to see what these conditions can look like in practice and to show that they are reasonable. The reader may choose to read it in parallel with the rest of the paper.

\subsection*{Acknowledgments}
The authors would like to thank Damaris Schindler, Philippe Michel, Jörg Brüdern, Rok Havlas and Julia Brandes  for stimulating discussions and feedback. The first author is supported by the DFG Research Training Group 2491 `Fourier Analysis and
Spectral Theory' and the second author is supported by SNSF grant 200021-197045. The authors are grateful to the University of Göttingen and EPFL for their hospitality.

\section{Set-up and notation}\label{S setup and notation}
We write $e(z) = e^{2\pi i z}$ and $L = \log (X)$. For $\alpha\in \RR$ define the exponential sum $$S(\alpha) := \sum_{0 \leq \bm x \leq X}a_{\bm x} e(\alpha f(\bm x)).$$ Then $R_{f, t}(X, \mathcal A) = \int_0^1 S(\alpha) d\alpha.$ We split up the interval $[0, 1]$ into major arcs, where we will show that $S(\alpha)$ is big, and minor arcs, where we will show that $S(\alpha)$ is small. For this we use the following major and minor arcs. Let $P$ and $Q$ be two fixed parameters and take $1 \leq a \leq q \leq P$ two integers such that $(a, q) = 1$. We define the major arc $\mathscr M(a, q)$ to be the interval $[\frac aq - \frac 1{qQ}, \frac aq + \frac 1{qQ}].$ The major arcs $\mathscr M$ are then defined as the union over $a, q$ as above, that is
$$\mathscr M = \cup_{1 \leq q \leq P}\cup_{\substack{1 \leq a \leq q \\ (a, q) = 1}} \mathscr M_{a, q}.$$
The minor arcs $\mathfrak m$ are the complement of $\mathscr M$ in $[0, 1]$. It will suffice to take  $Q=X^2P^{-1}$ and $P = L^B$ for some sufficiently large $B>0$ depending on $K$ and $c$ of Condition \ref{CL2}. In Section \ref{S Major arcs} we also consider the slightly enlarged major arcs $\mathscr{N}(a,q)$ being the interval $[\frac aq - \frac 1{Q}, \frac aq + \frac 1{Q}]$ and their union $\mathscr N$. 

Furthermore, we use the following notation throughout the entire paper:

Let $R$ be some set. For $\bm x = (x_1, \ldots, x_s) \in R^s$ and any function $\phi$ with domain $R$ we write $\phi(\bm x)$ for $\prod_{i=1}^s \phi(x_i)$. For example, for $\bm h=(h_1,\ldots,h_s) \in (\ZZ/q\ZZ)^s$ we have $\kappa(q,\bm h)=\prod_{i=1}^s\kappa(q,h_i).$ Similarly, when writing $0 \leq \bm x \leq X$ we mean $0 \leq x_i \leq X$ for all $i$. 

By $f(x) = O(g(x)), f(x) \ll g(x)$ and $g(x) \gg f(x)$ we mean that there exists some constant $C > 0$ such that for all $x$ sufficiently large $|f(x)| \leq C|g(x)|$. We write $f(x)\asymp g(x)$ if $f(x)\ll g(x)$ and $f(x)\gg g(x)$.

The symbol $\varepsilon$ denotes an arbitrarily small positive constant. The exact value of the constant may change between occurrences and even between lines.

We write $\sumst_{a\mods q}$ for the sum over integers $a$ between $1$ and $q$ that are coprime to $q$.

We will write $\|x\|$ to mean the minimal distance of $x$ to an integer, in other words $\|x\| = \min_{n\in \ZZ} |x - n|$. 

\subsection*{Smooth approximability} 
For the treatment of the major arcs and the interpretation of the singular integral we need to work with a smoothed version of the weight function. For these purposes the naive idea of a connecting the dots function is not always good enough. This subsection is devoted to giving a suitable definition of smooth approximability. 
\begin{defi}\label{Dsmthaprox}
    We say that the function $X\mapsto A(X)$ can be smoothly approximated if for every $m \in \NN$ there exists some smooth function $\Psi$ with derivative $\psi$ satisfying the following properties:
\begin{enumerate}
     \item $\Psi(X)=A(X)+O\left(\frac{A(X)}{L^m}\right),$
  \item $ \psi(X)\asymp \frac {A(X)}{X},$
  \item $\int_{yX}^{zX}|\psi'(u)|du \ll_{y, z} \frac{A(X)}
{X}, \textup{ where } y, z \in \RR_{>0}. $

\end{enumerate}
\end{defi}

\begin{example}
Some examples of such functions $\Psi(X)$ are $X^kL^{\ell}$ for all $k > 0$ and $\ell \in \RR$ and $X\log(\log(X)) + XL^{-1}$.

In Section \ref{SExample} we discuss the case where $\mathcal A$ is indicator function of the primes. 
\end{example}

\begin{lem}\label{Lsmthaprox}
    Assume that $A(X)$ can be smoothly approximated by $\Psi(X)$. Then $$\Psi(bX) \asymp_{y, z} \Psi(X) \text{ for all $b \in (y, z)$ where $y, z \in \RR_{> 0}$}.$$
\end{lem}

\begin{proof}
By Condition (1) and the monotonicity of $A(X)$ we have
$$\Psi(yX) \ll A(yX) \leq A(bX) \ll \Psi(bX) \ll A(bX) \leq A(zX) \ll \Psi(zX).$$
It therefore suffices to show that $\Psi(X) \ll_y \Psi(yX)$ and $\Psi(zX) \ll_z \Psi(X).$ We may furthermore assume that $y<1$ and $z>1$.
From Conditions (1) and (2) we know that there exist positive constants $c_i$ such that for $X$ big enough we have
$$c_1A(X)\leq \Psi(X)\leq c_2 A(X)\text{ and }|\psi(X)|\leq c_3\frac{A(X)}{X}.$$ 
We first consider the asymptotic upper bound for $\Psi(zX)$. Let $z_0>1$.
\begin{align*}
    \Psi(X) &= \Psi(z_0X) - \int_X^{z_0X} \psi(u) du\\
    &\geq \Psi(z_0X) - \int_X^{z_0X} |\psi(u)| du\\
    &\geq \Psi(z_0X) -  c_3 \int_X^{z_0X} \frac{A(u)}u du \\
    &\geq \Psi(z_0X) - c_3A(z_0X) \int_X^{z_0X} \frac{1}u du \\ 
    &\geq \Psi(z_0X) - \frac{c_3}{c_1}\Psi(z_0X) \log z_0\\
    &=\left(1-\frac{c_3}{c_1}\log z_0\right)\Psi(z_0X)
\end{align*}
Since $\frac{c_3}{c_1}$ is a positive constant, we can find $z_0>1$ such that $1-\frac{c_3}{c_1}\log z_0>0$. 
Since for all $z>1$ there exists a $k$ such that $z\leq z_0^k$ we can conclude by repeated application of the above inequality that 
\begin{equation*}\Psi(zX) \leq \left(\frac{1}{1-\frac{c_3}{c_1}\log z_0}\right)^k\Psi(X) \end{equation*}
and hence 
$$\Psi(zX)\ll_z\Psi(X).$$

Similarly we obtain for the lower bound with $y_0<1$ that
\begin{align*}
    \Psi(y_0X) &= \Psi(X) - \int_{y_0X}^{X} \psi(u) du\\
    &\geq \Psi(X) -  c_3 \int_{y_0X}^{X} \frac{A(u)}u du \\
    &\geq\left(1+\frac{c_3}{c_1}\log y_0\right)\Psi(X).
\end{align*}
Note that $\log y_0<0$.
Since $\frac{c_3}{c_1}$ is a positive constant, we can find $y_0<1$ such that $1+\frac{c_3}{c_1}\log y_0>0$. 
Since for all $y>1$ there exists a $k$ such that $y\geq y_0^k$ we can conclude by repeated application of the above inequality that 
\begin{equation*}\Psi(yX) \geq \left(1+\frac{c_3}{c_1}\log y_0\right)^k\Psi(X) \end{equation*}
and hence 
\begin{equation*}
\Psi(yX)\gg_y \Psi(X). \qedhere
\end{equation*}
    
\end{proof}

\section{A Weyl-like estimate and the minor arcs}\label{S weyl estimate}
In this section we will prove the following Weyl-like estimate and use it to bound the contribution of the minor arcs.
\begin{lem}\label{Lweyl}
Let $r := \rank(F_2)$ and $\alpha \in \RR$. Suppose Condition \emph{\ref{CL2}} holds and $\alpha = \frac aq + \lambda$ for some $a \in \ZZ$, $q \in \NN$ and $\lambda \in \RR$, then
    $$S(\alpha) \ll \max_{i, j}(F_{i, j})(A(X)L^C)^s\left(\frac 1X + \frac 1{q(1+X^2|\lambda|)} + \frac{q(1+X^2|\lambda|)}{X^2}\right)^{\frac r2},$$
    where the implicit constant depends only on $s$ and $r$ and $C$ is a constant depending on the constant $c$ in Condition \emph{\ref{CL2}}.
\end{lem}
\begin{rem}
    The proof is strongly based on the proof of Lemma 3.7 in \cite{Liu2011integralpoints}. For the sake of self-containedness we give a complete proof here.
\end{rem}

Recall that $n = \lfloor \frac s2 \rfloor$. We may express the vector $\bm x$ as $(\bm y, \bm z)$, where $\bm y$ is $n$-dimensional and $\bm z$ is $s-n$ dimensional. Then $f(\bm x)$ can be written as $$f(\bm x) = r(\bm y) + y_1g_1(\bm z) + \ldots + y_ng_n(\bm z) + q (\bm z),$$
where $r$ and $q$ are quadratic forms and $g_1, \ldots, g_n$ are linear forms. Similarly we can define linear forms $h_1, \ldots, h_{s-n}$ such that
$$f(\bm x) = r(\bm y) + z_1 h_1(\bm y) + \ldots + z_{s-n} h_{s-n}(\bm y) + q(\bm z).$$
\begin{lem}\label{Ldifferencing} Suppose we are in the situation of Lemma \ref{Lweyl}.
    We have 
    $$|S(\alpha)|^2 \ll (A(X)^2X^{-1}L^c)^{s}\sum_{\bm v \leq X} \prod_{j=1}^n\min\left(X, \frac{1}{\|\alpha g_j(\bm v)\|}\right),$$
    and
    $$|S(\alpha)|^2 \ll (A(X)^2X^{-1}L^c)^{s}\sum_{\bm v \leq X} \prod_{j=1}^{s-n}\min\left(X, \frac{1}{\|\alpha h_j(\bm v)\|}\right).$$
\end{lem}
\begin{proof}
    We rewrite the exponential sum $S(\alpha)$ as 
    $$S(\alpha) = \sum_{0 \leq \bm y \leq X}\sum_{0\leq \bm z \leq X} a_{\bm y}a_{\bm z} e(\alpha (f(\bm y, \bm z))).$$
    We can now start with applying Cauchy-Schwarz and Condition \ref{CL2} to find that
    \begin{align*}
    |S(\alpha)|^2 &\ll (A(X)^2X^{-1}L^c)^n \sum_{0 \leq \bm y \leq X}\left| \sum_{0 \leq \bm z \leq X} a_{\bm z}e(\alpha f(\bm y, \bm z))\right|^2 \\
    &= (A(X)^2X^{-1}L^c)^n \sum_{0 \leq \bm y \leq X}\sum_{0 \leq \bm z \leq X}\sum_{0 \leq \bm z' \leq X} a_{\bm z} a_{\bm z'} e(\alpha(f(\bm y, \bm z') - f(\bm y, \bm z))).
    \end{align*}
    We define $\bm v$ such that $\bm z' = \bm z + \bm v$ and note that 
    $$f(\bm y, \bm z + \bm v) - f(\bm y, \bm z) = \sum_{j=1}^n y_jg_j(\bm v) + q(\bm z + \bm v) - q(\bm z).$$
    Since $q(\bm z)$ and $q(\bm z + \bm v)$ do not depend on $\bm y$ and the absolute value of $e(\alpha(q(\bm z + \bm v) - q(\bm z))$ equals $1$ we find that
    $$|S(\alpha)|^2 \ll (A(X)^2X^{-1}L^c)^n\sum_{|\bm v|\leq X}\sum_{\bm z}a_{\bm z}a_{\bm z + \bm v}\left|\sum_{0 \leq \bm y \leq X} e\left(\alpha\sum_{j=1}^ny_jg_j(\bm v)\right)\right|,$$
    where the sum over $\bm z$ is taken such that both $\bm z$ and $\bm z + \bm v$ lie in the range $[0, X]$. Note that $\sum_{\bm z} a_{\bm z}a_{\bm z + \bm v} \leq \sum_{\bm z} a_{\bm z}^2$ and recall the standard inequality 
$$\sum_{0 \leq y \leq X}e(\alpha yg)\ll \min\left(X,\frac{1}{\|\alpha g\|}\right).$$
    Applying Condition \ref{CL2} a second time we get
    $$|S(\alpha)|^2 \ll (A(X)^2X^{-1}L^c)^{s}\sum_{|\bm v|\leq X} \prod_{j=1}^n \min\left(X, \frac{1}{\|\alpha g_j(\bm v)\|}\right).$$
    Analogously we can show that 
    \begin{equation*}
|S(\alpha)|^2 \ll (A(X)^2X^{-1}L^c)^s \sum_{|\bm v|\leq X} \prod^{s-n}_{j=1} \min\left( X, \frac 1 {\|\alpha h_j(\bm v) \|} \right).       
          \qedhere \end{equation*}
\end{proof}
\begin{lem} \label{Lstomu}
    Write $\bm x = (\bm y, \bm z)$ as before and define $H_j(\bm x) := g_j(\bm y)$ for all $1 \leq j \leq n$ and $H_j(\bm x) := h_{j-n}(\bm z)$ for all $n+1 \leq j \leq s$. 
    Let $\mu(X)$ be the number of $\bm x $ satisfying 
    $$|\bm x| \leq X,\ \|\alpha H_j(\bm x)\| \leq \frac 1X \text{ for all }1 \leq j \leq s$$
    Suppose that we are in the situation of Lemma \ref{Lweyl}. Then
    \[|S(\alpha)|^4 \ll A(X)^{4s}X^{-s}L^{2cs+s}\mu(X). \]
\end{lem}
\begin{proof}
    Without loss of generality we will only count the $\bm x$ within the box $[0, X]^s$ for which $\|\alpha H_j(\bm x)\| \leq \frac 1X$. This will lead to the constant $2^s$ appearing later on. For any integers $r_1, \ldots, r_s$ between $0$ and $X$ define $\mu_0(X, r_1, \ldots, r_s)$ to be the number of $\bm x$ inside $[0, X]^s$ such that 
    $$\{\alpha H_j(\bm x)\} \in \left[\frac{r_j}{X}, \frac{r_j + 1}{X}\right) \textup{ for all }1 \leq j \leq s.$$
    Then for any two points $\bm x$ and $\bm x'$ meeting these conditions we find that $$\|\alpha H_j(\bm x- \bm x')\| \leq \frac 1X$$ and we may conclude that $\mu_0(X, r_1, \ldots, r_s) \leq \mu(X)$ for any choice of $\bm r$. 
    By partitioning $[0, X]^s$ based on the value of $\bm r$ each element corresponds to, we find
    \begin{align*}
    \sum_{|\bm x| \leq X}\prod_{j=1}^s\min\left(X, \frac{1}{\|\alpha H_j(\bm x)\|}\right) & \ll 2^s\sum_{r_1 \leq X} \dots \sum_{r_s \leq X} \mu_0(X, r_1, \ldots, r_s)\prod_{j=1}^s \min\left(X, \frac{X}{r_j + 1}\right) \\
    & \ll \mu(X)\left(\sum_{r \leq X} \frac X{r+1}\right)^s \\
    & \ll \mu(X)(XL)^s.
    \end{align*}
The result now follows by Lemma \ref{Ldifferencing}, because 
\[|S(\alpha)|^4 \ll A(X)^{4s}X^{-2s}L^{2sc}\sum_{|\bm v|\leq X}\prod_{j=1}^s\min\left(X, \frac{1}{\|\alpha H_j(\bm v)\|}\right).\qedhere \]
\end{proof}
The next step will be to estimate $\mu(X)$. Note that $\mu(X)$ does not depend on our choice of $\mathcal A$ and therefore we can follow exactly the same argument as in \cite{Liu2011integralpoints}. The following lemma summarizes Lemma 3.4 and the beginning of the proof of Lemma 3.6 in \cite{Liu2011integralpoints}. 
\begin{lem} \label{Lmux}
Let $H_j$ and $\mu(X)$ be defined as in Lemma \ref{Lstomu} and $q \in \ZZ$, $X > 0$ and $\lambda \in \RR$. Define $$\rho := \frac{1}{4\|f\|}\min\left(1, \frac 1{qX|\lambda|}, \frac Xq ,\frac{q(1+X^2|\lambda|}X\right)$$ and let $\mu(\rho, X)$ be the number of $\bm x$ for which $|\bm x|\leq \rho X$ and $\|\alpha H_j(\bm x) \| \leq \frac \rho X$ for all $1 \leq j \leq s$. Then
$$\mu(X) \ll \frac{\mu(\rho, X)}{\rho^s} \ll (\rho X)^{-2r}X^s,$$
where the implicit constant only depends on $r$ and $s$. 
\end{lem}
The proof of Lemma \ref{Lmux} uses the geometry of numbers of Davenport, Lemma 12.6 in \cite{Davenport2005}, and Lemma 2.3 in \cite{HBcubicforms14var}. 

We are now ready to prove our Weyl-like estimate from the beginning of this section.
\begin{proof}[Proof of Lemma \ref{Lweyl}]
By Lemmata \ref{Lstomu} and \ref{Lmux} we know that
$$|S(\alpha)|^4 \ll (\rho X)^{-2r} A(X)^{4s}L^C.$$ 
The definition of $\rho$ as in Lemma \ref{Lmux} implies that
$$(\rho X)^{-2r} \ll \|f\|^{2r}\left(\frac 1X + \frac 1{q(1+X^2|\lambda|)} + \frac{1(1+X^2|\lambda|}{X^2}\right)^{2r}.$$
The result follows.
\end{proof}
\begin{lem}\label{Lminorarcs}
    Let $K > 0$ be fixed and suppose Conditions \emph{\ref{CL1}} and \emph{\ref{CL2}} both hold. Then, for $P = L^B$ with $B$ sufficiently large we find that 
    $$\int_\mathfrak m |S(\alpha)| d\alpha \ll A(X)^sX^{-2}L^{-K}.$$
\end{lem}

\begin{proof}
    Define $$E(K) := \bigcup_{1 \leq q \leq K} \bigcup_{\substack{a\mods q\\ (a, q)=1}}\left\{ \alpha \in [0, 1] \mid \left|\alpha - \frac aq\right| \leq \frac K {qX^2}\right\}.$$ 
    Then $E(P) = \mathscr M$ and $E(X) = [0, 1]$. Define $P_j = 2^jP$ for all non-negative integers $j$, then for some $k \ll L$ we find $X \leq P_k \leq 2X$. In general we have $E(P_{j-1}) \subset E(P_j)$. Hence, we may express the minor arcs as 
    $$ \mathfrak m = \bigcup_{j=1}^h E(P_j)\backslash E(P_{j-1}).$$
    Let $\alpha \in E(P_j)\backslash E(P_{j-1})$, then we may write $\alpha = \frac aq + \lambda$ such that
    $$\textup{either }~~P_{j-1} < q \leq P_{j}~~~\textup{ or }~~P_{j-1} < qX^2|\lambda| \leq P_j,$$
    where the right hand sides of the inequalities always hold. Hence,
    $$P_{j-1} < q(1 + X^2|\lambda|) \leq 2P_j.$$
    Combining this with Lemma \ref{Lweyl} we obtain
    $$\sup_{\alpha \in E(P_j)\backslash E(P_{j-1})} |S(\alpha)| \ll A(X)^sL^{Cs}\left( \frac 1X + \frac 1 {P_{j-1}} + \frac {P_j} {X^2} \right)^{\frac r2} \ll A(X)^sL^{Cs}P_j^{-\frac r2}.$$
    Hence 
    $$\int_{E(P_j)\backslash E(P_{j-1})} |S(\alpha)| \ll A(X)^sL^{Cs}P_j^{-\frac r2}\textup{Meas}(E(P_j)) \ll A(X)^sX^{-2}L^{Cs}P_j^{2-\frac r2}.$$
    Considering the total integral over the minor arcs, we find that
    $$\int_{\mathfrak m} |S(\alpha)| d\alpha \ll \sum_{j=1}^k A(X)^sX^{-2}L^{Cs}P_j^{2-\frac r2} \ll A(X)^sX^{-2}L^{Cs}P^{2-\frac r2}(1+ (2^{2-\frac r2})^L).$$ Choosing $P = L^B$ with $B$ sufficiently large completes the proof, because Condition \ref{CL1} implies that $2 - \frac r2 < 0$ holds. 
\end{proof}
\section{Treatment of the major arcs}\label{S Major arcs}
In this section we will turn our attention to the major arcs. The goal is to rewrite the integral as a product of a singular series and a singular integral. To simplify calculations, from now on we switch to the enlarged major arcs $\mathscr N$ as defined in Section \ref{S setup and notation}. Note that \begin{equation} \label{EQadjustmajorarcs}
    \int_\mathscr{M} S(\alpha)e(-\alpha t) d\alpha = \int_{\mathscr N} S(\alpha)e(-\alpha t) d\alpha + O(A(X)^sX^{-2}L^{-K}),
\end{equation}
because $\mathscr N \backslash \mathscr M$ is a subset of the minor arcs which by Lemma \ref{Lminorarcs} are bounded by the given error term. Hence, evaluating the integral over the bigger major arcs does not change our asymptotic. For integers $a$ and $q$ and a real number $\lambda$ we define
\begin{align*}
    S(q, a) &:= \sum_{\bm h \mods q} \kappa(q, \bm h) e\left(\frac aq f(\bm h)\right) \textup{~~~~ and}\\
I(\lambda,X) &:= \int_{[0,X]^s} \psi(\bm u)e(\lambda f(\bm u)) \text{d}\bm u.
\end{align*}
\begin{thm} \label{Tmajor}
Let $\alpha \in \mathscr N(a, q)$ such that $\alpha = \frac aq + \lambda$ and assume Condition \emph{\ref{CD}} holds with $Q_D \geq q$. Assume that $A(X)$ is smoothly approximable and let $m > 0$ be arbitrary. Then
$$S(\alpha) = S(q, a)I(\lambda,X) + O(q^s\lambda X^2 A(X)^{s}L^{-m}),$$
where the implicit constant may depend on $m$.
\end{thm}
\begin{proof}
    We may immediately write 
    \[S(\alpha) = \sum_{\bm h \mods q} e\left(\frac aq f(\bm h)\right) \sum_{\substack{ 0 \leq \bm x \leq X \\ \bm x \equiv \bm h \mods q}} a_{\bm x}e(\lambda f(\bm x)). \]
    We start by analyzing the second sum in one variable $x$. We replace $f(\bm x)$ by some $g(x) \in \ZZ[x]$ of degree $2$, not necessarily homogeneous. We then apply partial summation and find
    \begin{equation}\label{EQpartsum}
        \sum_{\substack{ 0 \leq x \leq X \\ x \equiv h \mods q}} a_{x}e(\lambda g(x)) = A(q, h; X)e(\lambda g(X)) - \int_0^X A(q, h; u) \text{d}e(\lambda g(u)).
    \end{equation}
    Because $A(X)$ can be smoothly approximated we may define $\Psi(X)$ as in Definition \ref{Dsmthaprox}, where the $m$ in Property (1) is taken equal to the one in our assumptions. By Condition \ref{CD} we then have that 
    $$A(q, h; u) = \kappa(q, h)A(u) + O(A(X)L^{-m}) = \kappa(q, h)\Psi(u) + O(A(X)L^{-m}).$$
    Substituting this into the integral of (\ref{EQpartsum}) and applying integration by parts we find 
    \begin{align*} \int_0^XA(q, &h; u) \text{d}e(\lambda g(u)) \\ &= \int_0^X \kappa(q, h)\Psi(u)\text{d}e(\lambda g(u)) + O(\lambda X^2A(X)L^{-m}) \\
    &= \kappa(q, h)\Psi(X)e(\lambda g(X)) - \int_0^X \kappa(q, h) e(\lambda g(u))\psi(u) \text{d}u + O(\lambda X^2A(X)L^{-m}).
    \end{align*}
    Substituting this back into (\ref{EQpartsum}) we get
    \[\sum_{\substack{ 0 \leq x \leq X \\ x \equiv h \mods q}} a_{x}e(\lambda g(x)) = \kappa(q, h)\int_0^X e(\lambda g(u))\psi(u) \text{d}u + O(\lambda X^2 A(X)L^{-m}).\]
    For the sum over all the variables $x_i$ we apply this equality $s$ times.
    \begin{align*}
        &\sum_{\substack{ 0\leq \bm x \leq X \\ \bm x \equiv \bm h \mods q}} a_{\bm x}e(\lambda f(\bm x)) \\ 
        & =\sum_{\substack{ 0 \leq x_1 \leq X \\ x_1 \equiv h_1}} a_{x_1} \cdots \sum_{\substack{ 0\leq x_s \leq X \\ x_s \equiv h_s}} a_{x_s} e(\lambda f(\bm x)) \\
        &  =\sum_{\substack{ 0 \leq x_1 \leq X \\ x_1 \equiv h_1}} a_{x_1} \cdots \sum_{\substack{ 0 \leq x_{s-1} \leq X \\ x_{s-1} \equiv h_{s-1}}} a_{x_{s-1}} \left(\kappa(q, h_s) \int_0^X e(\lambda f(\bm x))\psi(x_s) dx_s + O(\lambda X^2A(X)L^{-m})\right)\\
        & =\kappa(q, h_s) \int_0^X \psi(u) \sum_{\substack{ 0 \leq x_j \leq X \\ x_j \equiv h_j \\ 1 \leq j \leq s-1}} a_{x_1} \cdots a_{x_{s-1}} e(\lambda f((x_1,\ldots,x_{s-1},u))du + O(\lambda X^2A(X)^{s}L^{-m}) \\
        & =\kappa(q, \bm h) \int_{[0, X]^s} \psi(\bm u) e(\lambda f(\bm u)) d\bm u + O(\lambda X^2 A(X)^{s}L^{-m}).
    \end{align*}
    We conclude that indeed \[S(\alpha) = S(q, a)I(\lambda,X) + O\left(q^s\lambda X^2 A(X)^{s}L^{-m}\right). \qedhere\]
\end{proof}
Define the truncated singular series 
$$\mathfrak S(P) := \sum_{q\leq P} B(q) ~~~\textup{ with } ~~B(q) :=\sumst_{a \mods q} S(q, a)e\left(-\frac aq t\right)$$ 
and the truncated singular integral 
\begin{equation}\label{EQ truncated sing series}
   \mathfrak I(X) := \int_{-\frac 1Q}^{\frac 1Q}I(\lambda, X)e(-\lambda t) d\lambda. 
\end{equation}
Then, because $P \leq Q_D$, we find that
\begin{equation*}
    \int_{\mathscr N} S(\alpha)e(-\alpha t) d\alpha = \mathfrak S(P)\mathfrak I(X) + O(\textup{Meas}(\mathscr N) P^{s+1}A(X)^{s}L^{-m}).
\end{equation*}
We calculate that $\textup{Meas}(\mathscr{N}) \ll \frac{P^2}{Q} = \frac{L^{3B}}{X^2}$ and thus, using Condition \ref{CD} for $m$ sufficiently large, we may rewrite the error term to get
\begin{equation}\label{EQmajorint}
    \int_{\mathscr N} S(\alpha)e(-t\alpha) d\alpha = \mathfrak S(P)\mathfrak I(X) + O(X^{-2}A(X)^{s}L^{-K}),
\end{equation}
where $K$ can be taken as the same constant as before.
The next lemma is needed to replace $\mathfrak S(P)$ by its completed version.
\begin{lem}\label{Lbqbound}
Assuming Condition \emph{\ref{CK}} we have $$B(q) \ll q^{1-\frac r2+\varepsilon}.$$
\end{lem}

\begin{proof}
For each $q$ consider the sequence $\mathscr K_q := (\kappa(q, h))_{h \in \NN}$. We will use Condition \ref{CK} to show that $\mathscr K_q$ meets Condition \ref{CL2} for each $q$ and then apply Lemma \ref{Lweyl}. 

Let $X \geq q$, then

    \begin{align*}
         \sum_{x \leq X} \kappa(q, x)^2 
        &\ll \sum_{x\leq X}\left(\frac{(\log q)^c}{q}\right)^2 \\
        &=O\left(\frac{X(\log q)^{2c}}{q^2}\right).
    \end{align*}
 Furthermore, the right hand side of Condition \ref{CL2} becomes 
 $$\left(\sum_{x \leq X}\kappa(q, x)\right)^2X^{-1}L^c \leq \frac{X}{q^2}L^c + \frac{L^c}{X}.$$
 Note that the constant $c$ needed in Condition \ref{CL2} can be taken to be the constant $c$ from Condition \ref{CK}.
 Therefore the condition is satisfied. We can now use Lemma \ref{Lweyl} with $X=q$, $\alpha = \frac aq$, $\lambda = 0$ and sequence $\mathscr K_q$ to find the following bound:
    $$T\left( \frac aq \right) := \sum_{\bm h \mods q} \kappa(q, \bm h) e\left(\frac aq f(\bm h)\right) \ll 1^s (\log q)^{Cs}q^{-\frac r2}.$$
    Furthermore, 
    \begin{align*}
        B(q) &= \sumst_{a \mods q}\sum_{\bm h \mods q} \kappa(q, \bm h) e\left( \frac aq f(\bm h) \right) e\left( - \frac{a}{q}t\right) \\
        &= \sumst_{a \mods q}e\left( - \frac{a}{q}t\right) T\left(\frac aq\right).
    \end{align*}
    We may conclude that $|B(q)| \ll (\log q)^{Cs}q^{1-\frac r2}$.\qedhere
\end{proof}
 Assuming Condition \ref{CL1}, Lemma \ref{Lbqbound} implies that $B(q) \ll q^{-\frac{3}2 + \varepsilon}$ and therefore the sum over $B(q)$ converges. Define the (completed) singular series as
\begin{equation} \label{EQ untruncated sing series}
    \mathfrak S := \sum_{q} B(q).\end{equation}
 We get that
 \begin{equation}\label{EQsingserie}
     \mathfrak S - \mathfrak S(P) \ll P^{-1/2}.
 \end{equation}
 From this we may conclude that replacing $\mathfrak S(P)$ by $\mathfrak S$ in (\ref{EQmajorint}) does not change the main term of the asymptotic.

\section{Interpretation of the singular series}\label{S singular series}
In this section we relate the singular series to the counting problem. Our goal is to show that we can interpret it as a product of local densities outside of infinity. For this we first show that the completed singular series can be written as an Euler product. We then recall the definition of weighted $p$-adic densities and show that each factor in the Euler product is actually a weighted $p$-adic density.

We define the exponential sum including the shift by $t$ by 
$$ S_t(q, a) := \sum_{\bm h \mods q} \kappa(q, \bm h) e\left(\frac aq (f(\bm h)-t)\right),$$
where $\bm h=(h_1,\ldots, h_s)$. Then we have that
$$B(q)=\sumst_{\substack{a\mods q}}S_t(q,a).$$
\begin{lem}
    Assume Condition \emph{\ref{CC}}. Then the function $B(q)$ is multiplicative. 
\end{lem}
\begin{proof}
  To show the multiplicativity of $B(q)$ it suffices to show that for any coprime pair of natural numbers $(q,q')$ and any integers $a,a'$ we have 
    \begin{equation}\label{EQmult}
        S_t(q,a)S_t(q',a')=S_t(qq',qa'+q'a).
    \end{equation}
    We start by working on the right hand side
    $$S_t(qq',qa'+q'a)=\sum_{\bm h \mods {qq'}}\kappa(qq',\bm h)e\left(\frac {qa'+q'a}{qq'} (f(\bm h)-t)\right).$$
    Writing $\bm h=q'\bm h_1+q\bm h_2$ we can replace the sum over $\bm h\mods {qq'}$ by two sums over $\bm h_1 \mods q$ and $\bm h_2 \mods {q'}$ and obtain 
    $$\sum_{\bm h_1\mods q}\sum_{\bm h_2\mods {q'}}\kappa(qq',q'\bm h_1+q\bm h_2)e\left(\frac {qa'+q'a}{qq'}(f(q'\bm h_1+q\bm h_2)-t)\right).$$
The multiplicativity of $\kappa$ is given by Condition \ref{CC}. It hence matches the contribution of $\kappa$ on the left hand side. It is left to show that the two exponential sums are equal. 
Considering the contribution from the left hand side of (\ref{EQmult}), note that we can replace the summation over $\bm h_1 \mods q$ by $q'\bm h_1 \mods q$ and the summation over $\bm h_2 \mods {q'}$ by $q\bm h_2 \mods {q'}$ since $(q,q')=1$. 
Therefore it suffices to show that

$$e\left(\frac{q' a(f(q'\bm h_1)-t) + q a'(f(q\bm h_2) - t)}{qq'} \right) = e\left(\frac{(qa'+q'a)(f(q'\bm{h}_1 + q\bm h_2)-t)}{qq'} \right).$$
This can be done by considering each monomial of $f$ separately.
\end{proof}

We formally set
\begin{equation}\label{EQ defi chitilde}
\tilde{\chi}_p(\mathcal{A})=\sum_{k=0}^\infty B(p^k)\end{equation}
and, using the multiplicativity of $B(q)$ rewrite the completed singular series defined in (\ref{EQ untruncated sing series}) as an Euler product
\begin{equation} \label{EQprodlocdens}
\mathfrak{S}=\prod_p \tilde{\chi}_p(\mathcal{A}).
\end{equation}
Our goal is to show that this definition of $\tilde{\chi}_p(\mathcal A)$ matches the weighted $p$-adic density of solutions to $f(\bm x)=t$. To help with the intuition of weighted $p$-adic densities, we first recall the classical definition. 
\begin{defi}[classical $p$-adic density]
Let $Q$ be a subset of $\ZZ^s$. Consider the solutions to $r$ polynomial equations in $s$ variables:
$$V:=\{\bm{x}\in Q \mid \ f_1(\bm x)=t_1,\ldots,f_r(\bm x)=t_r\}.$$
Let 
$$V(\ZZ/p^m\ZZ):=\{\bm{x}\in Q/p^m\ZZ \mid \ f_1(\bm x)\equiv t_1,\ldots,f_r(\bm x)\equiv t_r\mods{p^m}\}.$$
Denote the number of solutions modulo $p^m$ by $v_m(p)$ and the number of inputs modulo $p^m$ by $q_m(p)$, 
\begin{align*}
    v_m(p)&:=|V(\ZZ/p^m\ZZ)| \\
    q_m(p)&:=|Q\mods{p^m}|.
\end{align*}
Note that there are $p^{mr}$ possible choices for $(t_1,\ldots,t_r)\mods{p^m}$. 
We then define the density of $V$ with respect to $Q$ modulo $p^m$ by $$\mu(p,m)=\frac{p^{mr}}{q_m(p)}v_m(p).$$
The $p$-adic density is defined as $$\chi_p(V):=\lim_{m\rightarrow \infty}\mu(p,m).$$
\end{defi}

In analogy to this we can now incorporate the weights. 

\begin{defi}[weighted $p$-adic density]
Let $Q$ be a subset of $\ZZ^s$. Consider the solutions to $r$ polynomial equations in $s$ variables:
$$V:=\{\bm{x}\in Q \mid \ f_1(\bm x)=t_1,\ldots,f_r(\bm x)=t_r\}.$$
Let there be weights for $x\mods{p^m}$ given by $\kappa(p^m,x)$.
Denote the weighted number of solutions modulo $p^m$ by $v_{\kappa,m}(p)$ and the number of weighted inputs modulo $p^m$ by $q_{\kappa,m}(p)$,
\begin{align*}
    v_{\kappa,m}(p)&:=\sum_{\bm x\in V(\ZZ/p^m\ZZ)}\kappa(p^m,\bm x)\\
    q_{\kappa,m}(p)&:=\sum_{\bm x\in Q\mods{p^m}}\kappa(p^m,\bm x).
\end{align*}
Note that there are still $p^{mr}$ possible choices for $(t_1,\ldots,t_r)\mods{p^m}$. 
We define the weighted density of $V$ with respect to $Q$ modulo $p^m$ by $$\mu_\kappa(p,m)=\frac{p^{mr}}{q_{\kappa,m}(p)}v_{\kappa,m}(p).$$
We define the weighted $p$-adic density as $$\chi_{\kappa,p}(V):=\lim_{m\rightarrow \infty}\mu_\kappa(p,m).$$\end{defi}

With this definition in mind we compute $\chi_{\kappa,p}(V)$ in our case. Observe that for us $r=1$ and hence the number of possible outcomes is $p^m$. Note furthermore that $Q=\ZZ^s$ and hence the size of $Q\mods {p^m}$ is given by 
\begin{align*}
 \sum_{\bm x \mods {p^m}}\kappa(p^m,\bm x)=1,
\end{align*}
since for all $1\leq i\leq s$ we have $\sum_{x_i\mods {p^m}}\kappa(p^m,x_i)=1$.
It is left to determine the size of $V(\ZZ/p^m\ZZ)$.

\begin{defi}    
Let $M(p^m)$ be the weighted count of solutions to our original counting problem modulo ${p^m}$, i.e. the $\bm x \in (\ZZ/p^m\ZZ)^s$ such that $f(\bm x)\equiv t \mods {p^m}$ weighted by $\kappa(p^m,\bm x)$. As a formula,
$$M(p^m)=\sum_{f(\bm x)\equiv t\mods {p^m} }\kappa(p^m,\bm x).$$
\end{defi}

Hence the weighted $p$-adic density is in our case given by \begin{equation}\label{EQ our weighted density}
\lim_{m\rightarrow \infty}p^mM(p^m).\end{equation} To conclude this section we show that (\ref{EQ our weighted density}) matches our definition of $\tilde{\chi_p}(\mathcal{A})$ in (\ref{EQ defi chitilde}).

\begin{lem} \label{LrelationChiandM}
The weighted $p$-adic density equals $\tilde{\chi}_p(\mathcal{A})$, i.e. 
    $$\tilde{\chi}_p(\mathcal{A})=\lim_{m\rightarrow\infty} p^mM(p^m).$$
\end{lem}

\begin{proof}
From $$\sum_{a\mods {p^m}}e\left(\frac{a}{p^m}x\right)=\begin{cases}
    p^m &\text{if }x\equiv0 \mods{p^m} \\
    0 &\text{else}
\end{cases}$$ 
we obtain by definition of $M(p^m)$ and $S_t(p^m,a)$ that 
$$p^m M(p^m)=\sum_{a\mods {p^m}}S_t(p^m,a).$$
We now group the numbers modulo $p^m$ by their greatest common divisor with $p^m$. 
$$p^m M(p^m)=\sum_{k=0}^m\sumst_{a\mods{p^k}}S_t(p^m,p^{m-k}a).$$
Now using that 
\begin{equation*}\label{kappa additive}    
\kappa(q,h)=\sum_{\substack{b\mods{qq'}\\b\equiv h\mods q}}\kappa(qq',b)
\end{equation*}
we have 
\begin{align*}
    S_t(p^m,p^{m-k}a)&=\sum_{\bm h \mods{p^m}} \kappa(p^m, \bm h) e\left(\frac {ap^{m-k}}{p^m} (f(\bm h)-t)\right)  \\
&= \sum_{\bm h \mods{p^k}} \sum_{\substack{\bm b\mods{p^m} \\ \bm b\equiv \bm h \mods{p^k} }}\kappa(p^m, \bm b) e\left(\frac a{p^k} (f(\bm h)-t)\right)\\   
&= \sum_{\bm h \mods{p^k}} \kappa(p^k, \bm h) e\left(\frac a{p^k} (f(\bm h)-t)\right)  \\
     &=S_t(p^k,a).
\end{align*}
Hence we obtain 
\begin{equation*}\lim_{m\rightarrow\infty} p^m M(p^m)=\lim_{m\rightarrow\infty}\sum_{k=0}^m\sumst_{a\mods{p^k}}S_t(p^k,a)=\sum_{k=0}^\infty B(p^k)=\tilde{\chi}_p(\mathcal{A}). \qedhere \end{equation*}

\end{proof}

This allows us to conclude that our singular series is indeed the product of the local densities outside of infinity. 

\begin{rem}
    In general it is not clear that this density is positive. We might need extra conditions, like for example Condition \ref{CB} in Section \ref{SExample} or the no bad primes condition in \cite{Liu2011integralpoints}. 
\end{rem}

\section{Interpretation of the singular integral}\label{S singular integral}
Recall that
\begin{align*}
\mathfrak I(X) = \int_{-1/Q}^{1/Q}I(\lambda,X)e(-\lambda t)\text{d}\lambda = \int_{-1/Q}^{1/Q}\int_{[0,X]^s} \psi(\bm u)e(\lambda f(\bm u)) e(-\lambda t)\text{d}\bm u\text{d}\lambda.
\end{align*}
We do a change of variables  $\bm u \rightarrow X\bm u, \lambda =  X^{-2}\gamma$ and $t_0=X^{-2}t$ to reduce the dependency on $X$ of the singular integral. Using that $f(X\bm u)=X^2f(\bm u)$ and $X^2/Q=P = L^B$, we obtain
\begin{align*}
 \mathfrak I(X) = &\int_{-P}^{P}\int_{[0,X]^s} \psi(X\bm u)e(X^{-2}\gamma f(X\bm u)) e(-X^{-2}\gamma X^2t_0)X^sX^{-2}\text{d}\bm u\text{d}\gamma  \\
= &~ \frac{\Psi(X)^s}{X^2} \int_{-P}^{P}\frac{X^s}{\Psi(X)^s}\int_{[0,1]^s} \psi(X\bm u)e(\gamma f(\bm u)) e(-\gamma t_0)\text{d}\bm u\text{d}\gamma.
\end{align*}
We would like to show three things. First, that integrating $\gamma$ over $(-\infty, \infty)$ instead of $(-P, P)$ gives an admissible error. Second, that this new integral grows asymptotically like $\frac{\Psi(X)^s}{X^s}$. Finally, that the constant in front of this asymptotic has an interpretation as the local density of solutions over the reals. This will be done in the next three subsections respectively. However, this only works if for $\bm x \in [0, 1]^s$ the equation $f(\bm x) = t_0$ has a solution and not too much curvature. To slightly generalize our results we consider the integral $\mathfrak I (X)$ over some box $\mathfrak B = \prod_{i=1}^s[y_i, z_i] \subset \RR_{>0}^s$ instead of $[0, 1]^s$. Hence, we define
\begin{equation}\label{EQsingintbox}
    \mathfrak I_{\mathfrak B}(X) := X^{s-2} \int_{-P}^{P}\int_{\mathfrak B} \psi(X\bm u)e(\gamma f(\bm u)) e(-\gamma t_0)\text{d}\bm u\text{d}\gamma.
\end{equation}
Furthermore, let
    $$\mathcal I(\gamma, \mathfrak B, X) := \int_{\mathfrak B} \psi(X\bm u) e(\gamma f(\bm u))d\bm u~~~~ \text{ and } ~~~~~\mathcal J(\gamma, \mathfrak B) :=  \int_{\mathfrak B} e(\gamma f(\bm u)) d\bm u.$$
We get that 
$$\mathfrak I_{\mathfrak B}(X) = X^{s-2} \int_{-P}^P \mathcal I(\gamma, \mathfrak B, X)e(-\gamma t_0) d\gamma.$$
The main consequence of changing to the box $\mathfrak B$ is that the weighted solutions we are now counting are those inside of $X\mathfrak B$, rather than $[0, X]^s$. 

\subsection*{Completing the singular integral}
In this part we complete the integral by first bounding $\mathcal J$ and then comparing $\frac {X^s}{\Psi(X)^s} \mathcal I$ with $\mathcal J$.
\begin{lem}\label{LupbndJ}
  Let $\gamma \in \RR$, and assume that Condition \textup{\ref{CL1}} holds, then
  $$\mathcal J(\gamma, \mathfrak B) \ll \min(\textup{Meas}(\mathfrak B), |\gamma|^{-2 - \varepsilon}).$$
\end{lem}

\begin{proof}
It is immediate that $\mathcal J$ is upper bounded by $\textup{Meas}(\mathfrak B)$ by taking the absolute value within the integral. For the second upper bound we consider the sequence $b_x = 1$ for all $x \in \NN$. This sequence meets Condition \ref{CL2} with $c=0$ and therefore we can apply Lemma \ref{Lweyl} on $T(\alpha) := \sum_{0 \leq \bm x \leq X} b_x e(\alpha f(\bm x))$ to see that 
    $$T(\lambda) \ll X^{s}L^{\frac s4}\left(\frac 1X + \frac 1{1+|\gamma|} + \frac {1 + |\gamma|}{X^2}\right)^{\frac r2},$$
where we have taken $a =0, q=1$ and $\gamma = X^2\lambda$ and, because $c=0$, we know $C \leq 1/4$. Furthermore, the function $X$ is a smooth approximation for the sequence $b_x$ and it has an error of $O(1)$. Therefore, by a similar argument as in the proof of Theorem \ref{Tmajor}, we have that 
    $$T(\lambda) = X^{s} \mathcal J(\gamma, \mathfrak B) + O(\gamma X^{s}L^{-m'}),$$
where we may take $m'$ as big as we like. For our purposes any $m' > \frac 32s$ will suffice, as we will see later on.
Combining the two facts about $T(\lambda)$ we find that 
    $$\mathcal J(\gamma, \mathfrak B) \ll |\gamma| L^{-m'} + L^{\frac s4} \left(\frac 1 X + \frac 1 {1+|\gamma|} + \frac{1+|\gamma|}{X^2}\right)^{\frac r2}.$$
Since the left hand side does not depend on $X$, we may choose $X$ depending on $\gamma$. Let $X$ be so that $|\gamma| = L^D$ where $D>0$ is some constant such that $\frac s4 - (\frac r2-2) D < 0$ and $m' > 3D$. We find that
    $$\mathcal J(\gamma, \mathfrak B) \ll L^{-2D - \varepsilon} = |\gamma|^{-2-\varepsilon}.$$
Note that in particular such a $D$ always exists if $m' > \frac {3s}{2r-8}$. This holds whenever $m' > \frac 32s$ by Condition \ref{CL1}.   
\end{proof}

\begin{lem}\label{Lcompsing}
Suppose that $\Psi(X)$ is the smooth approximation of $A(X)$ and that Condition \emph{\ref{CL1}} holds. Then
        $$\mathfrak I_{\mathfrak B}(X) = \frac{\Psi(X)^s}{X^2} \int_{-\infty}^{\infty} \frac{X^s}{\Psi(X)^s} \mathcal I(\gamma, \mathfrak B, X)e(-\gamma t_0) d\gamma + O\left(\frac{\Psi(X)^s}{X^2L^B}\right).$$
\end{lem}
\begin{proof}
Recall that
    $$\mathcal I(\gamma, \mathfrak B, X) = \int_{y_1}^{z_1}\psi(Xu_1) \cdots \int_{y_s}^{z_s} \psi(Xu_s) e(\gamma f(\bm u)) d\bm u.$$ 
    Define $$\Phi_1(\gamma, u_1) := \int_{y_1}^{u_1}\int_{y_2}^{z_2} \psi(Xu_2) \cdots \int_{y_s}^{z_s} \psi(Xu_s) e(\gamma f(w_1, u_2, \ldots, u_s)) du_s\cdots du_2dw_1.$$ Then, by partial integration in the variable $u_1$, we find
    $$\mathcal I(\gamma, \mathfrak B, X) = \left[\psi(Xu_1) \Phi_1(\gamma, u_1) \right]^{z_1}_{y_1} - \int_{y_1}^{z_1} X\psi'(Xu_1)\Phi_1(\gamma, u_1) du_1.$$
    Thus, 
    $$\mathcal I(\gamma, \mathfrak B, X) \ll \sup_{u_1 \in (y_1, z_1)}|\Phi_1(\gamma, u_1)|\left(|\psi(Xz_1)| + |\psi(Xy_1)| + \int_{y_1}^{z_1}|X\psi'(Xu_1)| du_1 \right).$$
    By the properties of Definition \ref{Dsmthaprox} and by Lemma \ref{Lsmthaprox} each term within the brackets is asymptotically bounded by $\Psi(X)X^{-1}$ and we have
    $$\mathcal I(\gamma, \mathfrak B, X) \ll \sup_{u_1 \in (y_1, z_1)} |\Phi_1(\gamma, u_1)| \Psi(X)X^{-1}.$$
    Let $1 \leq i \leq s$, write $\bm v_i$ for $(u_1, \ldots, u_{i-1}, w_i, u_{i+1}, \ldots, u_s)$ and define
    $$\Phi_i(\gamma, u_1, \ldots, u_i) := \int_{y_i}^{u_i}\int_{y_{i+1}}^{z_{i+1}} \psi(Xu_{i+1}) \cdots \int_{y_s}^{z_s} \psi(Xu_s) e(\gamma f(\bm v_i)) du_s \cdots du_{i+1}dw_i.$$
    In the following we will write $\Phi_i$ to mean $\Phi_i(\gamma, u_1, \ldots, u_i)$. 
    Let $2 \leq i \leq s$ then, by partial integration of the integral with respect to $u_{i}$, we get
    \begin{align*}
        \Phi_{i-1} = \int_{y_{i-1}}^{u_{i-1}} \left(\left[\psi(Xu_i) \Phi_i\right]_{u_i =y_i}^{z_i} -  
        \int_{y_i}^{z_i} X\psi'(Xu_i) \Phi_i du_i\right)dw_{i-1}.
    \end{align*}
    This shows that
    $$\Phi_{i-1}  \ll \int_{y_{i-1}}^{u_{i-1}}  \sup_{u_i \in (y_i, z_i)} |\Phi_i| \left(|\psi(Xz_i)| + |\psi(Xy_i)| + \int_{y_i}^{z_i}|X\psi'(Xu_i)|du_i\right) dw_{i-1},$$
    which by the same argument as above gives us
    $$\Phi_{i-1} \ll \int_{y_{i-1}}^{u_{i-1}} \sup_{u_i \in (y_i, z_i)} |\Phi_i| \Psi(X)X^{-1} dw_{i-1}.$$
    By induction we then find that for all $1 \leq i \leq s-1$ we have
    $$\sup_{u_1 \in (y_1, z_1)} |\Phi_1| \ll \Psi(X)^iX^{-i} \sup_{\substack{u_j \in (y_j, z_j) \\ \textup{for } 1 \leq j \leq i}}|\Phi_{i+1}|. $$
    Combining this with our estimate for $\mathcal I(\gamma, \mathfrak B, X)$ from above, we get
    $$\mathcal I(\gamma, \mathfrak B, X) \ll \sup_{\bm u \in \mathfrak B} |\Phi(\gamma, \bm u)|\Psi(X)^sX^{-s},$$
    where 
    $$\Phi(\gamma, \bm u) = \int_{y_1}^{u_1}\cdots \int_{y_s}^{u_s} e(\gamma f(\bm t)) dw_s \cdots dw_1.$$
    Lemma \ref{LupbndJ} showed that $\Phi(\gamma, \bm u) \ll \min(1, |\gamma|^{-2-\varepsilon})$, from which we deduce that 
    $$\mathcal I(\gamma, \mathfrak B, X) \ll \Psi(X)^sX^{-s}|\gamma|^{-2-\varepsilon}.$$
    Hence,
    $$\int_{P}^{\infty} X^s\Psi(X)^{-s} \mathcal I(\gamma, \mathfrak B, X) d\gamma \ll \int_{P}^{\infty} |\gamma|^{-2-\varepsilon} d\gamma \ll P^{-1}$$
    and the same holds for the integral between $-\infty$ and $-P$. This shows that
    \begin{equation*}\mathfrak I_{\mathfrak B}(X) = \frac{\Psi(X)^s}{X^2} \int_{-\infty}^{\infty} \frac{X^s}{\Psi(X)^s} \mathcal I(\gamma, \mathfrak B, X) e(-\gamma t_0)d\gamma + O\left(\frac{\Psi(X)^s}{X^2P}\right).\qedhere \end{equation*}
\end{proof}

\subsection*{Bounds for the singular integral}
Now that we have completed the singular integral we will show that it has the expected asymptotic upper and lower bounds of the form $\frac{\Psi(X)^s}{X^2}$. To do so we study the following object 
$$\tilde{\mathfrak I}_{\mathfrak B}(Y,X) =\int_{-Y}^Y \mathcal I(\gamma, \fr{B}, X) e(-\gamma t_0) d\gamma$$ 
for a suitable box $\mathfrak B$. In fact, our choice for $\mathfrak B$ will depend slightly on the height $X$, because it will depend on $t_0 = \frac t {X^2}$. The following lemma implies that counting the weighted solutions within the box $X\fr{B}$ is lower bounded and upper bounded by the total of weighted solutions in the boxes $X\fr{C}_1$ and $X\fr C_2$ respectively, where $\fr C_1$ and $\fr C_2$ do not depend on $X$. Since we assume $A(X)$ to be smoothly approximable, we know that asymptotically the weighted measure of the three boxes is the same. 

\begin{lem}\label{Lbox}
Assume that $f(\bm x) = 0$ has a solution in $\RR_{>0}^s$. Whenever $\eta > 0$ is sufficiently small there exist boxes $\mathfrak C_1$ and $\mathfrak C_2$ in $\RR_{>0}^s$, such that for all sufficiently large $X$ there is a solution  $\bm x^{\ast}\in\RR_{>0}^s$  to $f(\bm x) = t_0$  for which
   $\mathfrak C_1 \subset \bm x^{\ast} + [-\eta, \eta]^s \subset \mathfrak C_2$. 
\end{lem}
\begin{proof}
    Write $\bm x_0$ for the positive real solution to $f(\bm x) = 0$, that we have by our assumptions. Since $f$ is non-singular we know that whenever $\bm x$ is sufficiently close to $\bm x_0$ we have
    $$f(\bm x) = \sum_{i=1}^s\frac{\partial f}{\partial x_i}(\bm x_0) (x_i - x_{0, i}) + O(G_2),$$
    where $G_2$ is a degree $2$ form in $(\bm x - \bm x_0)$ and we may assume without loss of generality that $\frac{\partial f}{\partial x_1}(\bm x_0) \neq 0.$ Hence, for any $\delta > 0$ we can take $t_0$ sufficiently small such that there exists a solution $\bm x^\ast$ to $f(\bm x^\ast) = t_0$ and $|\bm x_0 - \bm x^{\ast}| < \delta$. Suppose that $\eta > 0$ is such that $\bm x_0 + [-2\eta, 2\eta]^s \subset \RR_{>0}^s$. Taking $\delta = \eta/2$ gives us, that whenever $X$ is sufficiently large, there exists a solution $\bm x^{\ast}$ to $f(\bm x^{\ast}) = t_0$ for which $|\bm x^{\ast} - \bm x| < \eta/2$. We define $$\mathfrak C_1 := \bm x_0 + [-\eta/2, \eta/2]^s, ~~~~~~~ \mathfrak C_2 := \bm x_0 + [-2\eta, 2\eta] ~~~~\textup{ and }~~~~\mathfrak B := \bm x^{\ast} + [-\eta, \eta]^s.$$ We will prove that $\mathfrak C_1 \subset \mathfrak B \subset \mathfrak C_2$. Let $\bm y _1 \in \mathfrak C_1$ and $\bm y_2 \in \mathfrak B$. We find
    $$|\bm x^{\ast} - \bm y_1| \leq |\bm x^{\ast} - \bm x_0| + |\bm x_0 - \bm y_1| \leq \eta/2 + \eta/2 = \eta$$
    and 
    $$|\bm x_0 - \bm y_2| \leq |\bm x_0 - \bm x^{\ast}| + |\bm x^{\ast} - \bm y_2| \leq \eta/2 + \eta < 2\eta.$$
    Thus, $\bm y_1 \in \mathfrak B$ and $\bm y_2 \in \mathfrak C_2$ and we are done.
\end{proof}

\begin{defi}\label{D boxB}
    Let $\bm x^{\ast}$ and $\eta > 0$ be as in Lemma \ref{Lbox}. We define $\mathfrak B:= \bm x^{\ast} + [-\eta, \eta]^s$.
\end{defi}

\begin{prop}\label{Pboundsing}
    Suppose that $A(X)$ is smoothly approximable and that $f(\bm x) = 0$ has a solution in $\RR^s_{>0}$. Let $X$ be fixed and define $\mathfrak B$ as above. Then
    $$I_0(X):= \lim_{Y \rightarrow \infty} \tilde{\mathfrak I}_{\mathfrak B}(Y,X) \asymp X^{-s} A(X)^s.$$
\end{prop}
\begin{proof}
We start with the expression for $\tilde{\mathfrak I}_{\mathfrak B}(Y,X)$ above and integrate over $\gamma$ to find that 
$$\tilde{\mathfrak I}_{\mathfrak B}(Y,X) = \int_{\mathfrak B} \frac{\sin(2\pi Y (f(\bm u) - t_0))}{\pi (f(\bm u)-t_0)} \psi(X\bm u) d\bm u.$$ 
Our first objective will be to rewrite our integral such that we can use Fourier's Theorem, see Section 9.7 \cite{Whittaker_Watson_1996}, which states
$$\lim_{Y \rightarrow \infty} \int_{[a, b]} \frac{\sin(2\pi Y w)}{\pi w} K(w) = K(0),$$ whenever $K(w)$ is a real function on $[a, b]$ of bounded variation.

Write $\bm x^{\ast} = (x_1^\ast,\ldots,x_s^\ast)$ for the non-singular positive real solution to the equation $f(\bm x) = t_0$, which we know exists by the proof of Lemma \ref{Lbox}. Then we know that
$$d_j:=\frac{\partial f}{\partial x_j}(\bm x^\ast)\neq 0\text{ for some }1\leq j \leq s.
$$Without loss of generality we assume that $j=1$ and $d_1=1$.

We do a change of variables to $\bm v = \bm u - \bm x^{\ast}$ such that our domain of integration changes to $[-\eta, \eta]^s$: 
$$ \tilde{\mathfrak I}_{\mathfrak B}(Y,X) =  \int_{[-\eta, \eta]^s} \frac{ \sin(2\pi Y (f(\bm x^{\ast} + \bm v) - t_0))}{\pi (f(\bm x^{\ast} + \bm v) - t_0)}\psi(X(\bm x^\ast + \bm v)) d\bm v .$$
To be able to replace $f(\bm x^{\ast} + \bm v)$ by something linear, we consider the Taylor expansion of $f$ near $\bm x^{\ast}$:
$$f(\bm x^\ast + \bm v) = v_1 + d_2 v_2 + \ldots + d_sv_s + G_2(\bm v),$$
where $G_2(\bm v)$ is a form of degree $2$. We define $w = f(\bm x^\ast + \bm v) - t_0$ and will do a change of variables from $v_1$ to $w$. By inverting the previous relation, we have that
$$ v_1 = t_0 + w - d_2v_2 - \ldots - d_sv_s + G(w, v_2, \ldots, v_s),$$
where $G$ is a multiple power series starting with terms of degree at least $2$. Hence, our Jacobian equals
$$\frac{\partial v_1}{\partial w}=1+G_1(w,v_2,\ldots,v_s)$$
for some multiple power series with terms of degree at least $1$. We will write $G_1$ instead of $G_1(w, v_2, \ldots, v_s)$ to keep our notation uncluttered. Since we have taken $\eta$ to be sufficiently small, we know that there exists some positive $\vartheta = \vartheta(\eta)$ such that whenever $v_1, \ldots, v_s$ are in $[-\eta, \eta]$ we have that $|w| < \vartheta(\eta)$ and that $|G_1|< \frac 12$. Whence,
$$\tilde{\mathfrak I}_{\mathfrak B}(Y,X) =  \int_{w \in [-\vartheta, \vartheta]} \frac{\sin(2\pi wY)}{\pi w} K(w) dw,$$
where 
\begin{align*}
K(w)=\int_{[-\eta, \eta]^{s-1}}(1+&G_1)\psi(X(x_2^\ast+v_2))\psi(X(x_3^\ast+v_3))\cdots \\ & \psi(X(x_s^\ast+v_s))\psi(X(x_1^\ast+w-d_2v_2-\ldots -d_sv_s+\varepsilon))dv_2\cdots dv_s,\end{align*}
and $\varepsilon$ comes from $G$. The function $K(w)$ has bounded variation, because $|\psi(X)| \geq c_1A(X)X^{-1}$ for some $c_1 > 0$ which shows that $\psi(X)$ cannot equal $0$ for any $X$ and thus cannot change sign. We can now apply Fourier's Theorem from which we conclude that $I_0 = K(0)$. Therefore it suffices to show that $K(0) \asymp A(X)^sX^{-s}.$ 

Because $\frac 12 < 1 + G_1 < \frac 32$ for all $v_2, \ldots, v_s$ we may ignore this factor. For the remaining factors we use the mean value theorem to find $c_2, \ldots, c_s \in [-\eta, \eta]$ such that
$$K(0) = (2\eta)^{2s} \psi(X(x_1^\ast - d_2c_2 - \ldots - d_sc_s +\varepsilon))\prod_{i=2}^s \psi(X(x_i^\ast + c_i)).$$
Write $c_1 = -d_2c_2 - \ldots - d_s c_s + \varepsilon$. We find that $\bm x^{\ast} + \bm c \in \mathfrak C_2$ and therefore, using the properties of being smoothly approximable and Lemma \ref{Lsmthaprox}, that 
$$I_0 \asymp \frac{\Psi(X(\bm x^{\ast} + \bm c))}{X(\bm x^{\ast} + \bm c)} \asymp \frac{\Psi(X)^s}{X^s} \asymp \frac{A(X)^s}{X^s}.$$
This concludes the proof. 
\end{proof}

\subsection*{Volume interpretation of the singular integral}
By Proposition \ref{Pboundsing} we can write $I_0(X)$ as $\sigma_\infty A(X)^sX^{-s}$, where $\sigma_\infty >0$ is depending on $X$ but can be lower and upper bounded independent of $X$. Combining this with Lemma \ref{Lcompsing} yields that the main term of $\mathfrak I_{\mathfrak B}(X)$ equals $\sigma_\infty A(X)^sX^{-2}.$ We will show that this $\sigma_\infty $ can be interpreted as the local density of real solutions to $f(\bm x) = t_0$ inside of $\mathfrak B$. This is done in a similar way as in Section 5 in \cite{biggs2023minimalistversioncirclemethod}. 

First, we change the measure to the $\mathcal{A}$-normalized measure
$$ d\mu_\mathcal{A}(z)=\frac{X}{\Psi (X) }\psi(Xz)dz.$$
We denote $d\mu_\mathcal{A}(\bm u)=\prod_{i=1}^sd\mu_\mathcal{A}(u_i)$. 
Note that this is a probability measure for all $X$, since $\mu_\mathcal{A}([0,1])=1$. 
For the singular integral we obtain 

$$\tilde{\mathfrak I}_{\mathfrak B}(Y,X)=\frac{\Psi(X)^s}{X^s} \int_{-Y}^Y\left(\int_\fr{B}e(\gamma(f(\bm u)-t_0))d\mu_\mathcal{A}(\bm u)\right)d\gamma.$$ 
Let 
\begin{align*}
    V_{\mathcal{A}}(h)&=\int_{\fr{B}}\mathds {1}_{[f(\bm u)-t_0=h]}d\mu_\mathcal{A}(\bm u)\\
    &=\vol \{\bm u\in \fr{B} : f(\bm u)-t_0=h\}.
\end{align*}
Here $\vol$ denotes the volume with respect to the measure $\mu_\mathcal{A}$. Note that $V_\mathcal{A}(0)$ counts the $\mu_\mathcal{A}$-weighted solutions to $f(\bm u)=t_0$. We now explain how to relate this to our singular integral. For this we consider the inverse Fourier transform of $V_\mathcal{A}$ at $\gamma$:
\begin{align*}
    \left(\mathcal{F}^{-1}V_\mathcal{A}\right)(\gamma)&=\int_\RR V_\mathcal{A}(h)e(\gamma h)dh \\
&=\int_\RR\int_{\fr{B}}\mathds {1}_{[f(\bm u)-t_0=h]}d\mu_\mathcal{A}(\bm u)e(\gamma h)dh \\
&=\int_\fr{B} e(\gamma(f(\bm u)-t_0))d\mu_\mathcal{A}(\bm u).
\end{align*}
Note that this is exactly the inner integral of $\tilde{\mathfrak I}_{\mathfrak B}(Y,X)$. Hence, now taking the forward Fourier transform on both sides and formally evaluating at 0, we obtain 

\begin{align*}
V_\mathcal{A}(0)&=\int_\RR\int_\fr{B} e(\gamma(f(\bm u)-t_0))d\mu_\mathcal{A}(\bm u)d\gamma \\
&=\frac{X^s}{A(X)^s} I_0(X) = \sigma_\infty . 
\end{align*}
This shows that we may interpret $\sigma_\infty $ as the volume of solutions to $f(\bm x) = t_0$ for $\bm x \in \mathfrak B$, where we have weighted the volume with $\mu_\mathcal{A}$. By smooth approximability we have $d\mu_{\mathcal A}(z) \asymp 1$ for all $z \in \mathfrak B$. This additionally shows that $\sigma_\infty$ is asymptotically lower and upper bounded by the non-weighted volume of solutions to $f(\bm x) = t_0$ for $\bm x \in \mathfrak B$.  

\begin{rem}
    As $X \rightarrow \infty$ we have that $t_0$ goes to $0$. Hence the box $\mathfrak B$, as defined in the proof of Lemma \ref{Lbox}, converges to $\bm x_0 +[-\eta, \eta]^s$. From the above we can then conclude that $\sigma_{\infty}$ converges to the weighted volume of solutions to $f(\bm x)=0$ inside of $\bm x_0 + [-\eta, \eta]^s$.
\end{rem}

\section{Proof of the main theorems}\label{S main thm}
\begin{proof}[Proof of Theorem \ref{Tmain}]
    It follows from Equations (\ref{EQ countingtocircle}), (\ref{EQadjustmajorarcs}), (\ref{EQmajorint}) and Lemma \ref{Lminorarcs} that 
    $$R_{f, t}(X, \mathcal A) = \mathfrak S(P)\mathfrak I(X) + O(A(X)^s X^{-2} L^{-K}).$$ 
    By Equation  (\ref{EQsingserie}) we may complete the singular series to get
    $$R_{f, t}(X, \mathcal A) = \mathfrak S \mathfrak I(X) + O(A(X)^sX^{-2}L^{-K} + \mathfrak I(X)L^{-\frac B2}).$$
    Here $\mathfrak S$ is the same constant as in Equation (\ref {EQ untruncated sing series}) which by Equation (\ref{EQprodlocdens}) and Lemma \ref{LrelationChiandM} can indeed be interpreted as the product of local densities over the $p$-adics for all the primes $p$. Taking $B \geq 2K$ completes the proof.
    \end{proof}
    \begin{proof}[Proof of Theorem \ref{Tmainbox}] 
    The proof that 
    $$R_{f, t}(X, \mathcal A, \mathfrak B) = \mathfrak S \mathfrak I_{\mathfrak B}(X) + O(A(X)^s X^{-2} L^{-K} + \mathfrak I_{\mathfrak B}(X)L^{-\frac B2})$$
    goes analogous to the proof above. It remains to show the results about $\mathfrak I_\mathfrak B(X)$.
    By Lemma \ref{Lcompsing} we may also complete this singular integral at the cost of $O(A(X)^sX^{-2}L^{-B})$. 
    Since $B \geq 2K$ we find
    $$R_{f, t}(X,\mathcal A, \mathfrak B) = \mathfrak SI_0 + O\left({A(X)^s}{X^{-2}} L^{-K} + I_0L^{-K}\right),$$
    and by Proposition \ref{Pboundsing} we know that 
    $$I_0 \asymp A(X)^sX^{-2}.$$  
    Finally, the interpretation of 
    $$ \frac{X^s}{A(X)^s}\mathfrak I_{\mathfrak B}(X)$$
    as the local density over the reals follows from the last subsection of Section \ref{S singular integral}. 
\end{proof}
\begin{rem}
    We only use Condition \ref{CC} to interpret the constant $ \mathfrak S$. Hence, it would also be possible to state a theorem without this condition that would give the same asymptotic formulae but without the local density interpretation. 
\end{rem}
\section{Examples}\label{SExample}
In this section we will go through two explicit examples for $\mathcal{A}$. Our goal is to illustrate the many conditions in Theorem \ref{Tmain} in action. We do not claim that the theorems below are in any sense optimal, since, by using the specific arithmetic structure of the sequences, one would expect to be able to get better error terms. In both cases we consider a non-singular quadratic form $f(\bm x)$ in $s$ variables that satisfies Condition \ref{CL1}. In particular it is assumed that $s\geq 10$.

For our first example we choose the weights to be the indicator function of primes. We show this example because it gives a non-trivial case of smooth approximability. 

For the second example we consider the indicator function of the $k$-free numbers. This example shows that our conditions on the sequence $\mathcal{A}$ are reasonable and gives some insight on the information about $\mathcal{A}$ that is needed.

Note that the case of squarefree numbers has already been investigated by Baker in \cite{Baker1} and \cite{Baker2}. He obtains a better error term and this would likely also be possible for $k$-free numbers. 

\subsection*{Primes}
If the weights $a_x$ are the indicator function of primes the prime number theorem gives us 
$$A(X) =\sum_{x\leq X}\mathds{1}_{\text{prime}}(x)= XL^{-1} \sum_{k=0}^{m-1} k! L^{-k} + O(A(X){L^{-m}}).$$
From this it is easy to check that for a given $m$ in Definition \ref{Dsmthaprox} the series $A(X)$ is smoothly approximable by $\Psi(X) = XL^{-1}\sum_{k=0}^{m-1} k!L^{-k}$. 

Furthermore we know that $\kappa$ is given by 
$$\kappa(q,a)=\begin{cases} 
    0&\text{ if $(q,a)>1$}\\
    \frac{1}{\varphi(q)} &\text{ else.}
\end{cases}$$
Here $\varphi$ denotes Euler's totient function. 
Conditions \ref{CD}, \ref{CC} and \ref{CK} are now straightforward to check. Note \cite{biggs2023minimalistversioncirclemethod}, Lemma 7.1 for Condition \ref{CC}. Furthermore Condition \ref{CL2} is satisfied as well, see Remark \ref{R L2withindicatorfct}.
This shows the following theorem. 

\begin{thm}
    \label{Tmainprimes}
Consider $f$ such that Condition \emph{\ref{CL1}} is satisfied and $f(\bm x)=0$ has a solution in $\RR^s_{>0}$. Fix $X$ and define $\mathfrak B$ as in Definition \ref{D boxB}.
Let $R_{f, t}(X, \mathcal A, \mathfrak B)$ be the number of solutions in primes in $X\mathfrak B$ to the equation $f(\bm x)=t$, i.e. 
$$R_{f, t}(X, \mathcal A, \mathfrak B)=\#\{\bm x=(x_1,\ldots,x_s)\in X\mathfrak B \mid  x_i \text{ is prime }\forall \ 1\leq i\leq s,\ f(\bm x)=t\}.$$
Let $K>0$ be any real number. Then
    $$R_{f, t}(X, \mathcal A, \mathfrak B)= \mathfrak S \sigma_\infty \frac{X^{s-2}}{L^s} \left( \sum_{k=0}^{K-1} k! L^{-k} \right)^s  + O(X^{s-2}L^{-K}).$$
    Here $\mathfrak S$ is a non-negative constant that can be interpreted as the product of $p$-adic densities and $\sigma_\infty$ is coming from the singular integral and can be interpreted as the local density of real solutions to $f(\bm x) = t$ inside of $X\mathfrak B$.
\end{thm}
\begin{rem}
    One can show positivity of the constant $\mathfrak S$ coming from the singular series assuming that $f$ is in some regard well behaved, for instance assuming the `no bad primes' condition in \cite{Liu2011integralpoints}. 
\end{rem}

\subsection*{$k$-free numbers}
Applying the machinery developed in this paper we obtain a special case of Theorem \ref{Tmain}. In the following we denote by $\zeta$ the Riemann zeta function.

\begin{thm}\label{Tmainkfree}
Consider $f$ such that Condition \emph{\ref{CL1}} is satisfied and $f(\bm x)=0$ has a solution in $\RR^s_{>0}$. Fix $X$ and define $\mathfrak B$ as in Definition \ref{D boxB}. Let $R_{f, t}(X, \mathcal A, \mathfrak B)$ be the number of solutions in $k$-free numbers in $X\mathfrak B$ to the equation $f(\bm x)=t$, i.e. 
$$R_{f, t}(X, \mathcal A, \mathfrak B)=\#\{\bm x=(x_1,\ldots,x_s)\in X\mathfrak B \mid  x_i \text{ is $k$-free }\forall \ 1\leq i\leq s,\ f(\bm x)=t\}.$$
    Let $K>0$ be any real number. Then
    $$R_{f, t}(X, \mathcal A, \mathfrak B) = \mathfrak S \sigma_\infty\frac{1}{\zeta(k)^s}X^{s-2} + O(X^{s-2}L^{-K}).$$
    Here $\mathfrak S \sigma_\infty$ is non-negative and can be interpreted as the product of local densities.
\end{thm}

We can furthermore provide a condition and argument for the singular series to be lower bounded. This allows us to conclude that the main term in the above theorem dominates the error term, i.e. that the constant $C$ is positive. Note that the required condition for this has to be tailored to the specific sequence $\mathcal{A}$. 
Our condition to ensure positivity of the singular series is analogous to Condition (B) in \cite{Baker1}. 

\begin{taggedcondition}{B}\label{CB}
     For all primes $p \mid 2\det F$ there exists a $\bm x=(x_1,\ldots,x_s)$ such that \begin{equation*}
        f(\bm x)\equiv t \mods {(2\det F)^{2k-1}}\text{ and } p^k\nmid x_i \text{ for all } i=1,\ldots,s.
    \end{equation*}  
\end{taggedcondition}

This allows us to conclude the following. 
\begin{thm}\label{Tkfree}
Let $f,\ \mathfrak B$ and $ R_{f, t}(X, \mathcal A, \mathfrak B)$ as in Theorem \ref{Tmainkfree} and assume Condition \emph{\ref{CB}}.
Let $K>0$ be any real number. Then
    $$R_{f, t}(X, \mathcal A, \mathfrak B)= \mathfrak S \sigma_\infty\frac{1}{\zeta(k)^s}X^{s-2} + O(X^{s-2}L^{-K}).$$
    Here $\mathfrak{S}\sigma_\infty$ is positive. The constant $\mathfrak S$ can be interpreted as the product of $p$-adic densities and $\sigma_\infty$ can be interpreted as the local density of real solutions to $f(\bm x) = t$ inside of $X\mathfrak B$.
\end{thm}

Note that Theorem \ref{Tmainkfree} and \ref{Tkfree} match Bakers results for the case of squarefree numbers.
More precisely \cite{Baker1}, Theorem 1 and 2 correspond to our first theorem and his Theorem 5 to our second. 

Before we start checking the required conditions we start with the general setup. 
A $k$-free number is by definition not divisible by $n^k$ for all $n\in \NN$. We denote their indicator function by 
$$\mu_k(n)=\begin{cases}
    1 &\text{ if $n$ is $k$-free} \\
    0 &\text{ else.}
\end{cases}$$
Note that for the squarefree numbers we have $\mu_2(n)=\mu(n)^2$.
In this example we choose our weights 
to be $a_x=\mu_k(x)$. 
Using the asymptotics of the $k$-free numbers as for example stated in \cite{bruedernetal}, Eq. (1.3), this yields 
$$A(X)=\sum_{x\leq X}\mu_k(x)=\frac{X}{\zeta(k)}+O(X^{\frac 1k}).$$
We hence set our error term to be $E(X)=X^{\frac 1k}$, which obviously satisfies $E(X)=O(A(X)L^{-m})$ for all $m\in \NN$.
Note furthermore that Condition \ref{CL2} is satisfied since our weights $a_x$ are an indicator function and $A(X)\asymp X$.

The next step is to find the asymptotics for $k$-free numbers in arithmetic progressions.  
More explicitly: for Condition \ref{CD} we need to find $\kappa(q,a)$ and a constant $Q_D=Q_D(X)$ such that for all $q\leq Q_D(X)$ and $h\mods q$ we have 
$$A(q,a;X)=\kappa(q,a)A(X)+ O(X^{1/k}).$$
The expression for $\kappa(q,a)$ can be found in \cite{bruedernetal}.
The definition given below is derived from their definition of $g$ around (2.6). 
Let the prime factorization of $q$ be $q=\prod_p p^{e_p}$.
Then  
$$\kappa(q,a)=\begin{cases}
    0&\text{ there exists $p$ s.t. }p^k\mid (q,a)\\
    \frac{1}{q}\prod_{p|q}(1-\frac{1}{p^k})^{-1}\prod_{\substack{p|q\\p\nmid \frac{q}{(q,a)}}}(1-p^{e_p-k})&\text{ else.}
\end{cases}$$
In \cite{bruedernetal}, Eq. (2.13) the stated error term of $O(X^{1/k})$ can be found. It follows that we can take $Q_D(X)=X^{1-\frac 1k}$.
\begin{rem}
    The error term $O(X^{1/k})$ for $k$-free numbers in arithmetic progressions is not the best possible. It was already known to Prachar \cite{Prachar1958} for the case of coprime residue class and modulus. It has since been improved, but this is irrelevant for us.  
\end{rem}

Condition \ref{CC} follows from noting that for $(q,q')=1$ we have $$(qq',qh'+q'h)=(q,q'h)(q',qh').$$

It is left to show Condition \ref{CK}, i.e. that there exists a constant $c\in \NN$ such that $\kappa(q,a)\ll \frac{(\log q)^c}{q}$.
If $\kappa(q,a)$ vanishes the statement is trivial. Note furthermore that $e_p-k\leq -1$. Hence $(1-p^{e_p-k})<1$ for all $p$. We obtain 
\begin{align*}    
\kappa (q,a)&\leq \frac 1q\prod_{p|q}(1-\frac{1}{p^k})^{-1} \\
&\leq \frac 1q \left(\frac{2^k}{2^k-1}\right)^{\omega(q)} \\
&\leq \frac 1q \left(\frac{4}{3}\right)^{\omega(q)}
\end{align*}
where $\omega(q)$ counts the number of distinct primes dividing $q$. 
By the Hardy-Ramanujan Theorem we know that 
$$\omega(q)=\log\log q + O(\log \log (q) ^{1/2+\varepsilon}.$$
Hence $\kappa(q,a)\ll \frac{\log q}{q}$ and Condition \ref{CK} holds with $c=1$.
Having checked all the conditions we can hence conclude Theorem \ref{Tmainkfree}. 

\begin{proof}[Proof of Theorem \ref{Tkfree}]
The argument to show positivity assuming Condition \ref{CB} is an adaptation of 

the proof of Theorem 5 in \cite{Baker1} to $k$-free numbers. It uses a slight modification of \cite{Davenport2005}, Lemma 17.1 to show the following. 

\begin{lem}\label{L number of sol}
    For all primes $p$ there exists $1\leq \beta\leq k$ and $\alpha\geq \beta$ such that for $\nu\geq 0$ there are at least $p^{(s-1)\nu}$ solutions to 
    $$f(\bm x)\equiv t\mods{p^{\alpha+\nu}} \text { and } x_i\not\equiv 0\mods{p^\beta} \text{ for all }1\leq i\leq s.$$
    
\end{lem}    
The argument to conclude that the singular series is bounded from below is then analogous to \cite{Liu2011integralpoints}, Proof of Prop. 2.2 or equivalently \cite{Davenport2005}, Thm. 17.1.
\end{proof}

\printbibliography
\end{document}